\documentclass[a4paper]{amsart}
\usepackage[latin1]{inputenc}
\usepackage{amssymb}
\usepackage{amsmath}
\usepackage{mathrsfs}
\usepackage{eufrak}
\usepackage{amsthm}
\usepackage{amsfonts}
\usepackage{textcomp}
\usepackage{graphicx}
\usepackage[pdftex]{color}
\usepackage{paralist}
\usepackage[shortlabels]{enumitem}
\usepackage{hyperref}
\usepackage{comment}
\usepackage[arrow, matrix, curve]{xy}

\newtheorem*{corollary*}{Corollary}
\newtheorem{theorem}{Theorem}[section]

\newtheorem{corollary}[theorem]{Corollary}

\newtheorem{proposition}[theorem]{Proposition}
\newtheorem{question}[theorem]{Question}

\newtheorem*{claim*}{Claim}
\newtheorem*{conjecture}{Conjecture}

\theoremstyle{definition}

\newtheorem*{theorem }{Theorem}
\newtheorem*{conjecture }{Conjecture}

\newtheorem{example}[theorem]{Example}

\theoremstyle{remark}

\numberwithin{equation}{theorem}

\makeatletter
\renewcommand*\env@matrix[1][\
arraystretch]{%
  \edef\arraystretch{#1}%
  \hskip -\arraycolsep
  \let\@ifnextchar\new@ifnextchar
  \array{*\c@MaxMatrixCols c}}
\makeatother

\renewcommand{\mod}{\operatorname{mod}}

\newcommand{\Ext}{\operatorname{Ext}}
\newcommand{\Tor}{\operatorname{Tor}}
\newcommand{\pd}{\operatorname{pd}}
\newcommand{\id}{\operatorname{id}}
\newcommand{\Tr}{\operatorname{Tr}}
\newcommand{\TF}{\operatorname{TF}}
\newcommand{\Dom}{\operatorname{Dom}}

\newcommand{\End}{\operatorname{End}}
\newcommand{\ev}{\operatorname{ev}}
\newcommand{\Hom}{\operatorname{Hom}}
\newcommand{\add}{\operatorname{\mathrm{add}}}

\renewcommand{\mod}{\operatorname{mod}}

\newcommand{\domdim}{\operatorname{domdim}}

\newcommand{\gldim}{\operatorname{gldim}}

\begin{document}

\title{A bimodule approach to dominant dimension}
\date{\today}

\subjclass[2010]{Primary 16G10, 16E10}

\keywords{Gorenstein projective modules, Nakayama conjecture, dominant dimension,reflexive modules}

\author{Ren\'{e} Marczinzik}
\address{Institute of algebra and number theory, University of Stuttgart, Pfaffenwaldring 57, 70569 Stuttgart, Germany}
\email{marczire@mathematik.uni-stuttgart.de}

\begin{abstract}
We show that a finite dimensional algebra $A$ has dominant dimension at least $n \geq 2$ if and only if the regular bimodule $A$ is $n$-torsionfree if and only if $A \cong \Omega^{n}(\Tr(\Omega^{n-2}(V)))$ as $A$-bimodules, where $V=\Hom_A(D(A),A)$ is the canonical $A$-bimodule in the sense of \cite{FKY}. We apply this to give new formulas for the Hochschild homology and cohomology for algebras with dominant dimension at least two and show a new relation between the first Tachikawa conjecture, the Nakayama conjecture and Gorenstein homological algebra.
\end{abstract}

\maketitle
\section*{Introduction}
The dominant dimension of a finite dimensional algebra $A$ with a minimal injective coresolution $(I_i)$ of the regular module $A$ is defined as the smallest $n$ such that $I_n$ is not projective or infinite in case no such $n$ exists. One of the most important homological conjecture for finite dimensional algebras is the Nakayama conjecture, introduced by Nakayama in 1958, that states that an algebra $A$ has infinite dominant dimension if and only if $A$ is selfinjective. We refer to the survey \cite{Yam} for the Nakayama and related conjectures.
On the other hand, the dominant dimension also has very practical applications. Namely, by the Morita-Tachikawa correspondence, see for example \cite{Ta} chapter 10, an algebra $A$ has dominant dimension at least two if and only if $A$ has a minimal faithful projective-injective module with the double centraliser property. This can be used to provide computation-free proofs of the Schur-Weyl duality or Soergel's double centraliser theorem for blocks of category $\mathcal{O}$, see \cite{KSX}. Recently, the dominant dimension was also used in the definition of higher Auslander algebras and the correspondence with cluster tilting modules by Iyama, see \cite{Iya}, generalising Auslander's classical homological characterisation of representation-finite algebras. For a recent survey on the dominant dimension and applications, we refer to \cite{Koe}.

In this article we give a new bimodule characterisation of the dominant dimension. Of special interest is when an algebra has dominant dimension at least two, since this property is equivalent to having the double centraliser property with respect to a minimal faithful projective-injective module. For this reason we state our main theorem first in this special case.
Recall that an $A$-module $M$ is called reflexive in case the natural evaluation map $\ev_M: M \rightarrow M^{**}$ is an isomorphism, where $N^{*}:=\Hom_A(N,A)$ denotes the $A$-dual of a module $N$.
It is a classical result from linear algebra that every finite dimensional vector space is reflexive and more generally every projective module over an algebra is reflexive, however not every module is reflexive and it is an interesting problem to give a classification of reflexive modules over certain rings, see for example \cite{E}.
Letting $A^e=A \otimes_K A^{op}$ denote the enveloping algebra of $A$, it is well known that $A$-bimodules correspond to $A^e$-modules. Following \cite{FKY}, the canonical bimodule $V$ of $A$ is defined as $V:= \Hom_A(D(A),A)$.
We will see that $V \cong A^{*}$ as right $A^e$-modules, where $A^{*}$ is the $A^e$-dual of the bimodule $A$.
\begin{theorem }
Let $A$ be a finite dimensional algebra. Then the following are equivalent:
\begin{enumerate}
\item $A$ has dominant dimension at least two.
\item $A$ as a bimodule is reflexive.
\item $A$ is a 2-syzygy module as a bimodule, that is $A \cong \Omega^2(N)$ for some $A$-bimodule $N$.
\item $A \cong \Omega^2(\Tr(V))$ as $A$-bimodules.

\end{enumerate}

\end{theorem }
Here $\Tr(X)$ of a module $X$ denotes the Auslander-Bridger transpose of $X$.
The authors in \cite{FKY} used the canonical bimodule $V$ to give for the first time a characterisation of algebras with dominant dimension at least two via a bimodule isomorphism.
Namely, they showed that an algebra $A$ has dominant dimension at least two if and only if $D(A) \cong D(A) \otimes_A V \otimes_A D(A)$ as $A$-bimodules. It is a natural question whether one can characterise higher dominant dimension via a similiar bimodule isomorphism using the canonical bimodule $V$. Our next result will reveal such a bimodule isomorphism condition.
In \cite{AB}, Auslander and Bridger generalised the notion of being reflexive for a module by saying that an $A$-module $M$ is $n$-torsionfree if $\Ext_A^i(\Tr(M),A)=0$ for $i=1,...,n$. Then a module $M$ is reflexive if and only if it is $2$-torsionfree.
Another useful tool for the homological theory of noetherian rings introduced by Auslander and Bridger in \cite{AB} are the functors $J_k:= \Tr \Omega^k$ that we call higher Auslander-Bridger translates in the following. Using these notions we can generalise the previous theorem as follows:
\begin{theorem }
Let $A$ be a finite dimensional algebra and $n \geq 2$.
Then the following are equivalent:
\begin{enumerate}
\item $A$ has dominant dimension at least $n$.
\item $A$ as a bimodule is $n$-torsionfree.
\item $A$ as a bimodule is an $n$-th syzygy module, that is $A \cong \Omega^n(N)$ for some $A$-bimodule $N$.
\item $A \cong \Omega^{n}(J_{n-2}(V))$ as $A$-bimodules.
\end{enumerate}

\end{theorem }

The previous theorem gives a new viewpoint of the Nakayama conjecture in terms of the recently introduced $\mho$-quiver of an algebra $A$ by Ringel and Zhang, see \cite{RZ}, by showing that $A$ has infinite dominant dimension if and only if $A$ as a bimodule is the end point of an infinite path in the $\mho$-quiver of the enveloping algebra of $A$.

The previous theorem also gives a new connection between the dominant dimension and Gorenstein homological algebra.
Gorenstein homological algebra can be viewed as a generalisation of classical homological algebra and its main concern is the classification of Gorenstein projective modules and the singularity category for Gorenstein algebras. We refer for example to the book \cite{EJ} and the survey \cite{Che} for more on Gorenstein homological algebra.
Recall that an $A$-module $M$ is called Gorenstein projective in case $\Ext_A^i(M,A)=0=\Ext_A^i(\Tr(M),A)$ for all $i>0$. 
We pose the following new homological conjecture in this article, that we call the Gorenstein bimodule conjecture:
\begin{conjecture }
Let $A$ be a finite dimensional algebra. Then $A$ as a bimodule is Gorenstein projective if and only if $A$ is selfinjective.

\end{conjecture }
The question when $A$ as a bimodule is Gorenstein projective seems to be first considered by Shen in \cite{S}, where it was proven that in case $A$ is Gorenstein, the Gorenstein bimodule conjecture is true.
More generally, in \cite{Mar}, the conjecture was shown to be true for left weakly Gorenstein algebras.
Recall that the first Tachikawa conjecture states that in case $\Ext_A^i(D(A),A)=0$ for all $i>0$ then $A$ is selfinjective, see for example \cite{Yam} for this conjecture and related conjectures. Note that by results of Mueller in \cite{Mue} the truth of the Nakayama conjecture for all algebras would imply the truth of the first Tachikawa conjecture for all algebras, but for a fixed algebra it is not known whether the truth of the Nakayama conjecture implies the truth of the first Tachikawa conjecture. For example any local non-selfinjective algebra has dominant dimension zero and thus the Nakayama conjecture holds for such algebras, while it is not known whether the first Tachikawa conjecture holds for local algebras.
As a corollary of our main result we can relate the Gorenstein bimodule conjecture to the first Tahchikawa conjecture and the Nakayama conjecture.
\begin{theorem }
Let $A$ be a fixed finite dimensional algebra.
Then $A$ satisfies the Gorenstein bimodule conjecture if and only if $A$ satisfies the first Tachikawa conjecture or the Nakayama conjecture.
\end{theorem }
As an application of the previous theorem, we will see that the Gorenstein bimodule conjecture holds for all finite dimensional algebras with finite finitistic dimension.
The connection between the Nakayama and the Gorenstein bimodule conjecture is particular strong for gendo-symmetric algebras $A$, where we will see that $A$ has infinite dominant dimension if and only if $A$ as a bimodule is Gorenstein projective. 
We profited from experimenting with the GAP-package QPA that motivated several results of this article, see \cite{QPA}.
\section{Preliminaries}
Unless otherwise stated we assume that $A$ is a non-semisimple connected finite dimensional algebra over a field $K$ and modules are finitely generated left modules.
$\mod-A$ denotes the module category of $A$ and $\underline{\mod}-A$ the stable module category. $D(-):=\Hom_K(-,K)$ denotes the natural duality of a finite dimensional algebra $A$ over a field $K$.
We give a quick summary of definitions for the dominant dimension and $n$-torsionfree modules, where we refer for example to \cite{Ta} and \cite{AB} for more details.
The \emph{enveloping algebra} of $A$ is defined as $A^e:=A \otimes_K A^{op}$ and it is well known that the module category of $A^e$ is equivalent to the category of $A$-bimodules.
A module $M$ with minimal injective coresolution $(I_i)$ is defined to have \emph{dominant dimension} $\domdim(M)$ equal to $n$, when $n$ is the smallest integer such that $I_n$ is not projective, or infinite in case no such $n$ exists.
The dominant dimension of the algebra $A$ is defined as the dominant dimension of the regular module and it is well known that this also coincides with the dominant dimension of the right regular module and the dominant dimension of $A$ as a bimodule over the enveloping algebra $A^e$, see \cite{Mue}.
Following \cite{FKY}, we call $V:=\Hom_A(D(A),A)$ the \emph{canonical bimodule}.
Algebras with dominant dimension at least one are also called \emph{QF-3} algebras and this condition is equivalent to the existence of a minimal faithful projective-injective $A$-module of the form $eA$ for some idempotent $e$.
The \emph{Morita-Tachikawa correspondence} states that an algebra $A$ has dominant dimension at least two if and only if $A \cong \End_B(M)$ for an algebra $B$ with a generator-cogenerator $M$ of $\mod-B$. In this case $B \cong eAe$ when $eA$ is the minimal faithful projective-injective $A$-module and $B$ is called the \emph{base algebra} of $A$. $A$ is called gendo-symmetric in case $A$ has dominant dimension at least two and the base algebra $B$ of a $A$ is a symmetric algebra, that is $B \cong D(B)$.
We will need the following results about gendo-symmetric algebras:
\begin{theorem} \label{Fankoetheorem}
Let $A$ be a finite dimensional algebra. 
\begin{enumerate}
\item $A$ is gendo-symmetric if and only if $V \cong A$.
\item In case $A$ is gendo-symmetric, $\domdim(A) = \inf \{ i \geq 1 | \Ext_A^i(D(A),A) \neq 0 \}+1$.

\end{enumerate}

\end{theorem}
\begin{proof}
For (1), see theorem 3.2. of \cite{FanKoe2}. For (2) see proposition 3.3. of \cite{FanKoe2}. 

\end{proof}

For more on gendo-symmetric algebras we refer to \cite{FanKoe} and \cite{Mar2}.
$M$ is called an \emph{$n$-th syzygy module} in case $M \cong \Omega^n(N)$ for some other $A$-module $N$.
The \emph{$A$-dual} of an $A$-module $M$ is defined as $M^{*}:=\Hom_A(M,A)$, which is a right $A$-module. Since $(-)^{*}=\Hom_A(-,A)$ is a functor, we also have $A$-duals of $A$-linear maps.
For an $A$-module $M$, the \emph{evaluation map} $\ev_M: M \rightarrow M^{**}$ is defined by $\ev_M(m)(g)=g(m)$, when $g \in M^{*}$.
$M$ is called \emph{torsionless} in case $\ev_M$ is injective, which is equivalent to $M$ being a submodule of a projective module or equivalently a first syzygy module. $M$ is called \emph{reflexive} in case $\ev_M$ is an isomorphism, which is equivalent to $M \cong M^{**}$ as $A$-modules. Following Auslander and Bridger in \cite{AB} the \emph{Auslander-Bridger transpose} of a module $M$ with minimal projective presentation $P_1 \xrightarrow{f} P_0 \rightarrow M \rightarrow 0$ is defined as the cokernel of the map $P_0^{*} \xrightarrow{f^{*}} P_1^{*}$. 
The \emph{higher Auslander-Bridger tranpose} of a module $M$ is defined as $J_{n}(M):=\Tr(\Omega^{n}(M))$ for $n \geq 0$.
We call $M$ \emph{$n$-torsionfree} in case $\Ext_A^i(\Tr(M),A) =0$ for $i=1,...,n$. Being 1-torsionfree is the same as being torsionless and being 2-torsionfree is the same as being reflexive.
Following \cite{RZ}, for a module $M$ one defins $\mho(M)$ as the cokernel of a minimal left $\add(A)$-approximation of $M$.

The \emph{$\mho$-quiver} of a finite dimensional algebra $A$ is defined as the quiver having vertices the isomorphism classes of indecomposable non-projective $A$-modules $[X]$ and there is an arrow $[\mho(X)] \rightarrow [X]$ for any torsionless indecomposable non-projective module $X$.

We collect some results on $\mho$.
\begin{proposition} \label{ringelzhangpropo}
Let $A$ be a finite dimensional algebra.
\begin{enumerate}
\item $\mho^k \cong \Tr \Omega^k \Tr$ for $k \geq 1$.
\item In case $M$ is torsionless, $\mho(M)$ is indecomposable and not projective and $\Omega( \mho(M)) \cong M$.
\item $[M]$ is the start of a path of length $t \geq 1$ in the $\mho$-quiver if and only if $\Ext_A^i(M,A)=0$ for $i=1,...,t$.
\item $[M]$ is the end of a path of length $t \geq 1$ if and only if $M$ is $t$-torsionfree.

\end{enumerate}

\end{proposition}
\begin{proof}
\begin{enumerate}
\item See the lemma in 4.4. of \cite{RZ}.
\item This is a consequence of lemma 3.3. in \cite{RZ}.
\item See (1) of the theorem in section 1.5. of \cite{RZ}.
\item See (2) of the theorem in section 1.5. of \cite{RZ}.
\end{enumerate}

\end{proof}

$\Dom_n(A)$ denotes the full subcategory of modules having dominant dimension at least $n$, $\TF_n(A)$ denotes the full subcategory of $n$-torsionfree modules and $\Omega^n(\mod-A)$ the full subcategory of $A$-modules that are $n$-th syzygy modules.

\begin{theorem} \label{FKYtheorem}
Let $A$ be a finite dimensional algebra.
\begin{enumerate}
\item $A$ has dominant dimension at least one if and only if there is an $A$-bimodule monomorphism $ A \rightarrow \Hom_A(D(A), \Hom_A(\Hom_A(D(A),A),A)$.
\item $A$ has dominant dimension at least two if and only if there is an $A$-bimodule isomorphism  $A \cong \Hom_A(D(A), \Hom_A(\Hom_A(D(A),A),A)$.

\end{enumerate}

\end{theorem}
\begin{proof}
\begin{enumerate}
\item See \cite{FKY}, proposition 4.6.
\item See \cite{FKY}, theorem 4.7.

\end{enumerate}

\end{proof}

We will need the following two general results:
\begin{proposition} \label{torformula}
Let $A$ be a finite dimensional algebra with two $A$-modules $M,N$.
\begin{enumerate}
\item $\Tor_i^A(M,N) \cong D \Ext_A^i(M,D(N))$ for all $i \geq 0$.
\item In case $P$ is a projective $A$-bimodule, then $P \otimes_A M$ is a projective $A$-module.

\end{enumerate}
\end{proposition}
\begin{proof}
\begin{enumerate}
\item See for example proposition 4.11. in appendix A of \cite{ASS}.
\item See for example lemma 11.15. in chapter IV. in \cite{SkoYam}.

\end{enumerate}

\end{proof}
\begin{theorem} \label{ausreisyzygy}
Let $A$ be an algebra of dominant dimension at least $n \geq 1$.
Then for $1 \leq i \leq n$: $\Dom_i(A)=\TF_i(A)=\Omega^i(\mod-A)$ and $\TF_{n+1}(A)=\Omega^{n+1}(\mod-A)$.

\end{theorem}
\begin{proof}
By proposition 1.6. (b) of \cite{AusRei} we have that $\Omega^i(\mod-A)=\TF_i(A)$ for all $i=1,..,n+1$ in case the subcategory $\Omega^i(\mod-A)$ is extension-closed for $i=1,...,n$.
By theorem 0.1. of \cite{AusRei} $\Omega^i(\mod-A)$ is extension-closed for $i=1,...,n$ if and only if the flat dimensions of the modules $I_i$ are less than or equal to $i+1$ for $i<n$. Since $A$ having dominant dimension $n$ implies that $\pd(I_i)=0$ for $i<n$ and thus also that their flat dimensions (which coincides with the projective dimensions for finitely generated modules over finite dimensional algebras) are less than or equal to $i+1$ and thus 
for $1 \leq i \leq n+1$: $\Omega^i(\mod-A)=\TF_i(A)$. Now we have also $\Omega^i(\mod-A)=\Dom_i(A)$ for $1 \leq i \leq n$ by \cite{MarVil}, proposition 4.
\end{proof}

\begin{proposition} \label{domdimtorsyzprop}
Let $A$ be a finite dimensional algebra and $M$ an $A$-module.
Consider the following conditions:
\begin{enumerate}
\item $A$ has dominant dimension at least $n$ and $M$ has dominant dimension at least $n$.
\item $M$ is $n$-torsionfree.
\item $M$ is an $n$-th syzygy module.

\end{enumerate}
Then we have that (1) implies (2) and (2) implies (3).
\end{proposition}
\begin{proof}
That (1) implies (2) follows directly by \ref{ausreisyzygy} and that (2) implies (3) holds in general, see for example at the end of page 7 in \cite{AB}. 

\end{proof}
We remark that in general we do not have that (3) implies (2) or that (2) implies that $M$ has dominant dimension at least $n$ since in general $A$ might not even have projective-injective non-zero modules.
We give a quick example of a module that is a 2-syzygy module but not reflexive (=2-torsionfree):
\begin{example}
Let $A=K[x,y]/(x^2,y^2,xy)$, then $A$ is a 3-dimensional local algebra with simple module $S$.
Then $U:= \Omega^2(S)$ is a 4-dimensional module that is an $2$-th syzygy module.
But $\Hom_A(\Hom_A(U,A),A)$ is 16-dimensional and thus $U$ can not be reflexive.

\end{example}

\section{A new characterisation of the dominant dimension of algebras}
In this section we prove our main result using induction. We first prove the result for the small cases $n=1$ and $n=2$ separately.

\begin{proposition} \label{isomorphismspropo}
Let $A$ be a finite dimensional algebra.
We have the following isomorphisms of $A$-bimodules for an $A$-bimodule $X$:
\begin{enumerate}
\item $A^e \cong \Hom_K(D(A),A)$.
\item $\Hom_{A^e}(X, A^e) \cong \Hom_A(D(A) \otimes_A X, A)$,. In particular for $X=A: V=\Hom_A(D(A),A) \cong \Hom_{A^e}(A,A^e)=A^{*}$, the $A^e$-dual of $A$.
\item $\Hom_{A^e}(\Hom_{A^e}(A,A^e),A^e) \cong \Hom_A(D(A),\Hom_A(\Hom_A(D(A),A),A))$.
\end{enumerate}

\end{proposition}
\begin{proof}
\begin{enumerate}
\item See corollary 4.4. of \cite{AusRei}
\item See corollary 4.2. of \cite{AusRei}.
\item By (2), we get $\Hom_{A^e}(A, A^e) \cong \Hom_A(D(A), A)$
and thus $\Hom_{A^e}(\Hom_{A^e}(A,A^e),A^e) \cong \Hom_{A^e}(\Hom_A(D(A),A),A^e)$. Now setting $X=\Hom_A(D(A),A)$ in (2), we get $ \Hom_{A^e}(\Hom_A(D(A),A),A^e) \cong \Hom_A(D(A) \otimes_A \Hom_A(D(A),A) , A)$.
Then the adjoint isomorphism between Hom and the tensor product gives us
$\Hom_A(D(A) \otimes_A \Hom_A(D(A),A) , A) \cong \Hom_A(D(A),\Hom_A(\Hom_A(D(A),A),A))$.

\end{enumerate}

\end{proof}

\begin{proposition} \label{propodomdim1}
Let $A$ be a finite dimensional algebra.
Then the following are equivalent:
\begin{enumerate}
\item $A$ is a QF-3 algebra.
\item $A$ as a bimodule is torsionless.
\item $A$ has a bimodule is a syzygy of another bimodule.

\end{enumerate}

\end{proposition}
\begin{proof}
That (1) implies (2) and (2) implies (3) is clear by \ref{domdimtorsyzprop}.
Now assume that $A$ is a syzygy module, which is equivalent to $A$ being torsionless (recall that being torsionless is equivalent to being a syzygy module for a general module).
Then the evaluation map $\ev_A: A \rightarrow \Hom_{A^e}(\Hom_{A^e}(A,A^e),A^e)$ is a monomorphism, but by 
\ref{isomorphismspropo} (3), we have an $A$-bimodule isomorphism $g: \Hom_{A^e}(\Hom_{A^e}(A,A^e),A^e) \cong \Hom_A(D(A),\Hom_A(\Hom_A(D(A),A),A))$.
Thus the composition of maps $ g \circ \ev_A :A \rightarrow \Hom_A(D(A),\Hom_A(\Hom_A(D(A),A),A))$ is an $A$-bimodule monomorphism, which implies that $A$ has dominant dimension at least one by \ref{FKYtheorem} (1).
\end{proof}

\begin{theorem} \label{reflexivetheo}
Let $A$ be a finite dimensional algebra. Then the following are equivalent:
\begin{enumerate}
\item $A$ has dominant dimension at least two.
\item $A$ as a bimodule is reflexive.
\item $A$ is a 2-syzygy module as a bimodule, that is $A \cong \Omega^2(N)$ for some $A$-bimodule $N$.
\item $A \cong \Omega^2(\Tr(V))$ as $A$-bimodules.

\end{enumerate}

\end{theorem} 
\begin{proof}
That (1) implies (2) and (2) implies (3) is clear by \ref{domdimtorsyzprop}.
Now we show that (3) implies (1). 
So assume that $A$ is a 2-syzygy module. Then $A$ is a syzygy module and by \ref{propodomdim1} $A$ has dominant dimension at least one. 
By \ref{ausreisyzygy} $A$ as an $A$-bimodule is 2-torsionfree, which is equivalent to being reflexive.
Thus the evaluation map $\ev_A: A \rightarrow \Hom_{A^e}(\Hom_{A^e}(A,A^e),A^e)$ is an isomorphism. But by \ref{isomorphismspropo} (3) we have that $\Hom_{A^e}(\Hom_{A^e}(A,A^e),A^e) \cong \Hom_A(D(A),\Hom_A(\Hom_A(D(A),A),A))$ and thus there is an isomorphism of $A$-bimodules $A \rightarrow \Hom_A(D(A),\Hom_A(\Hom_A(D(A),A),A))$ which by \ref{FKYtheorem} (2) implies that $A$ has dominant dimension at least two.
Thus (1), (2) and (3) are equivalent. That (4) implies (3) is trivial.
Now assume (2) and we show that (4) holds.
By \ref{isomorphismspropo} we have that $V=\Hom_A(D(A),A) \cong \Hom_{A^e}(A, A^e)$.
Let $P_1 \rightarrow P_0 \rightarrow V \rightarrow 0$ be a minimal projective resolution of $V$ as an $A$-bimodule.
By the definition of the Auslander-Bridger transpose of $V$ we get the following exact sequence:
$$0 \rightarrow V^{*} \rightarrow P_0^{*} \rightarrow P_1^{*} \rightarrow \Tr(V) \rightarrow 0.$$
Here $(-)^{*}$ denotes the application of the functor $\Hom_{A^e}(-,A^e)$.
Since $P_0^{*}$ and  $P_1^{*}$ are projective again and since we assume that $A$ is reflexive as an $A$-bimodule, we have $V^{*} \cong A^{**} \cong A$ and $A \cong \Omega^2(\Tr(V))$.

\end{proof}

We now give a higher dimensional generalisation of \ref{reflexivetheo} using the following recent result of Luo and Zhang that we state here only in a special case that we need:
\begin{theorem} \label{domdimcharatheorem}
Assume $A$ has dominant dimension at least two.
Then 
$$\domdim(A) = \sup \{ i \geq 1 | \Ext_A^i(D(A) \otimes_A \Hom_A(D(A),A) ,A) \neq 0 \}+1.$$

\end{theorem} 
\begin{proof}
This is theorem 4.2. of \cite{LZ} in the special case $X=A$ and noting that any algebra of dominant dimension at least two satisfies $\domdim(D(A)^{**}) \geq 2$, see remark 4.3. in \cite{LZ}.

\end{proof}

We now come to the proof of our main result:
\begin{theorem} \label{mainresult}
Let $A$ be a finite dimensional algebra and $n \geq 2$.
Then the following are equivalent:
\begin{enumerate}
\item $A$ has dominant dimension at least $n$.
\item $A$ as a bimodule is $n$-torsionfree.
\item $A$ as a bimodule is an $n$-th syzygy module, that is $A \cong \Omega^n(N)$ for some $A$-bimodule $N$.
\item $A \cong \Omega^{n}(J_{n-2}(V))$ as $A$-bimodules.

\end{enumerate}

\end{theorem}
\begin{proof}
That (1) implies (2) and (2) implies (3) is clear by \ref{domdimtorsyzprop}.
Now we show that (3) implies (1). 
We use induction to show that (3) implies (1). The case $n=2$ is true by \ref{reflexivetheo}.
Assume the result is true for $n$, we then show it is also true for $n+1$ for $n \geq 2$.
Thus assume that $A$ is an $(n+1)$-th syzygy module.
Then $A$ is especially an $n$-th syzygy module and by induction $A$ has dominant dimension at least $n$. Now by \ref{ausreisyzygy} $A$ is also $(n+1)$-torsionfree, which is equivalent to
$\Ext_{A^e}^i(\Tr(A),A^e)=0$ for $i=1,...,n+1$.
Let $A^{*}=\Hom_{A^e}(A,A^e)$ denote the dual of $A$ as an $A$-bimodule and let 
$G_1 \rightarrow G_0 \rightarrow A \rightarrow 0$ be a minimal projective presentation of $A$ as an $A$-bimodule.
By definition of the Auslander-Bridger transpose of $A$ we have the following exact sequence:
$$0 \rightarrow A^{*} \rightarrow G_0^{*} \rightarrow G_1^{*} \rightarrow \Tr(A) \rightarrow 0.$$
Thus $A^{*} \cong \Omega^2(\Tr(A)$ and thus 
$$\Ext_{A^e}^i(A^{*},A^e)=\Ext_{A^e}^i(\Omega^2(\Tr(A)),A^e) \cong \Ext_{A^e}^{i+2}(\Tr(A),A^e).$$
Now let $V=A^{*} \cong \Hom_A(D(A),A)$ and 
\begin{align} \label{formula0}
P_{n-1} \rightarrow \cdots \rightarrow P_0 \rightarrow V \rightarrow 0
\end{align}
be the beginning of a minimal projective $A$-bimodule resolution of $V$.
Since we know that $A$ as a bimodule is $n+1$-torsionfree and thus $\Ext_{A^e}^i(A^{*},A^e)=\Ext_{A^e}^{i+2}(\Tr(A),A^e)=0$ for $i=1,...,n-1$, we obtain the following exact sequence:
\begin{align} \label{formula1}
0 \rightarrow V^{*} \rightarrow P_0^{*} \rightarrow \cdots P_{n-1}^{*}
\end{align}
where $(-)^{*}=\Hom_{A^e}(-,A^e)$ denotes the $A^e$-dual.
Now by \ref{isomorphismspropo} (2), we have for all $A$-bimodules $U: 
\Hom_{A^e}(U,A^e) \cong \Hom_A(D(A) \otimes_A U, A)$.
Using this in \ref{formula1} , we obtain 
the following exact sequence:
$$ 0 \rightarrow \Hom_A(D(A) \otimes_A V,A) \rightarrow \Hom_A(D(A) \otimes_A P_0,A) \rightarrow \cdots \Hom_A(D(A) \otimes_A P_{n-1},A) \rightarrow \cdots .$$
On the other hand tensoring the minimal projective resolution \ref{formula0} with $D(A)$ gives the following beginning of a projective resolution of $A$-modules:
\begin{align} \label{formula2}
D(A) \otimes_A P_{n-1} \rightarrow \cdots \rightarrow D(A) \otimes_A P_0 \rightarrow D(A) \otimes_A V \rightarrow 0.
\end{align}
Here we used two things, first that for a general projective $A$-bimodule $P$ we have that $T \otimes_A P$ is a projective $A$-module for any $A$-module $T$ by \ref{torformula} (2). Second, we used that $\Tor_i(D(A),V)=0$ for all $i=1,...,n-1$ so that \ref{formula2} is really exact. To see this note that $\Tor_i^{A^e}(D(A),V) \cong D \Ext_{A^e}^i(D(A),D(V)) \cong D \Ext_{A^e}^i(V,A)=D \Ext_{A^e}^i(A^{*},A)=0$ for $i=1,...n-1$, using \ref{torformula} (1).
This shows that in case we have $\Ext_{A^e}^{i+2}(\Tr(A),A^e)=\Ext_{A^e}^i(A^{*},A^e)=0$ for $i=1,...,n-1$, we also have $\Ext^i(D(A) \otimes_A V ,A)= \Ext^i(D(A) \otimes_A \Hom_A(D(A),A) ,A)=0$ and by \ref{domdimtorsyzprop} this shows that $\domdim(A) \geq n+1$. Thus (1), (2) and (3) are equivalent. 
Now assume (4), then we trivially have (3).
Assume (2) now and we want to show (4), thus assume that $A$ is $n$-torsionfree as an $A$-bimodule, that is $\Ext_{A^e}^i(\Tr(A),A^e)=0$ for $i=1,...,n$ and $n \geq 2$.
Let 
$$P_{n-1} \rightarrow P_{n-2} \rightarrow \cdots P_0 \rightarrow V \rightarrow 0$$
be a minimal projective resolution of $V$ and apply the functor $\Hom_{A^e}(-,A^e)$ to it to obtain the exact sequence
\begin{align} \label{formula3}
0 \rightarrow V^{*} \rightarrow P_0^{*} \rightarrow \cdots P_{n-2}^{*} \rightarrow P_{n-1}^{*},
\end{align}
where this sequence is exact since $V \cong A^{*}$ and 
$$\Ext_{A^e}^i(V,A^e)=\Ext_{A^e}^i(A^{*},A^e)=\Ext_{A^e}^i(\Omega^2(\Tr(A)),A^e)=\Ext_{A^e}^{i+2}(\Tr(A),A^e)=0$$
for $i=1,...,n-2$ by assumption.
Since we have a minimal projective presentation 
$$P_{n-1} \rightarrow P_{n-2} \rightarrow \Omega^{n-2}(V) \rightarrow 0$$
by definition,  
the cokernel of the map $P_{n-2}^{*} \rightarrow P_{n-1}^{*}$ is equal to $\Tr(\Omega^{n-2}(V))$.
Since all terms $P_i^{*}$ are projective, we obtain from \ref{formula3} that $V^{*} \cong \Omega^{n}(\Tr(\Omega^{n-2}(V)))$.
Since we assume $n \geq 2$, $A$ is reflexive and thus $V^{*} \cong A^{**} \cong A$, which shows that $A \cong \Omega^{n}(\Tr(\Omega^{n-2}(V)))$.
\end{proof}

We apply the previous theorem to give new formulas for the Hochschild cohomology and homology of finite dimensional algebras. For the definition and basic properties of the Hochschild homology and cohomology we refer for example to \cite{W}.
Recall that the functor $\tau_{n-1}:=D \Tr \Omega^{n-2}= \tau \Omega^{n-2}$ is called \emph{higher Auslander-Reiten translate}, following \cite{Iya}.
\begin{proposition} \label{formulahochschild}
Let $A$ be a finite dimensional algebra with dominant dimension $n \geq 2$.
\begin{enumerate}
\item We have for the Hochschild homology and $l \geq 1$:
$$HH_l(A) \cong D \Ext_{A^e}^{l+n}(A, \tau_{n-1}(V)).$$
\item We have for the Hochschild cohomology and $l \geq 1$:
$$HH^l(A) \cong \Ext_{A^e}^{l+n}(D(A), \tau_{n-1}(V)).$$

\end{enumerate}
\end{proposition}
\begin{proof}
By \ref{mainresult} we have that in case $A$ has dominant dimension $n$: $A \cong \Omega^n(\Tr(\Omega^{n-2}(V))$.
\begin{enumerate}
\item 
Now $HH_i(A)=\Tor_i^{A^e}(A,A) \cong D \Ext_{A^e}^i (A,D(A))$, using \ref{torformula} (1).
Using $A \cong \Omega^n(\Tr(\Omega^{n-2}(V))$, we obtain 
$$\Ext_{A^e}^i (A,D(A)) \cong \Ext_{A^e}^i (\Omega^n(\Tr(\Omega^{n-2}(V)),D(A)) \cong \Ext_{A^e}^{i+n} (\Tr(\Omega^{n-2}(V)),D(A)) \cong \Ext_{A^e}^{i+n}(A, \tau_{n-1}(V)).$$
\item Here 
$$HH^i(A)= \Ext_{A^e}^i(A,A) \cong \Ext_{A^e}^i (\Omega^n(\Tr(\Omega^{n-2}(V)),A) \cong \Ext_{A^e}^{i+n} (\Tr(\Omega^{n-2}(V)),A) \cong \Ext_{A^e}^{i+n}(D(A), \tau_{n-1}(V)).$$

\end{enumerate}
\end{proof}

Since the formulas are especially nice for gendo-symmetric algebras, which are exactly those algebras with $V \cong A$ as $A^e$-modules, we state this as a corollary:
\begin{corollary}
Let $A$ be a gendo-symmetric algebra with dominant dimension $n \geq 2$.
\begin{enumerate}
\item We have for the Hochschild homology and $l \geq 1$:
$$HH_l(A) \cong D \Ext_{A^e}^{l+n}(A, \tau_{n-1}(A)).$$
\item We have for the Hochschild cohomology and $l \geq 1$:
$$HH^l(A) \cong \Ext_{A^e}^{l+n}(D(A), \tau_{n-1}(A)).$$

\end{enumerate}

\end{corollary}

In the next corollary we remark that \ref{formulahochschild} can be used to give vanishing results for the Hochschild homology and cohomology for algebras with dominant dimension at least two.

\begin{corollary}
Let $A$ be a finte dimensional algebra over a field $K$ with dominant dimension $n \geq 2$.
\begin{enumerate}
\item We have $HH_l(A)=0$ for all $l > \pd_{A^e}(A)-n$.
\item We have $HH^l(A)=0$ for all $l > \pd_{A^e}(D(A))-n$.

\end{enumerate}

\end{corollary}
\begin{proof}
(1) and (2) follow immediately from the previous proposition.

\end{proof}

We remark that part (1) can be used to obtain a quick proof that the Hochschild homology of all higher Auslander algebras over an algebraically closed field vanishes in positive degrees since we have $\pd_{A^e}(A)=\gldim(A)$ by work of Happel in \cite{Ha} in case the field is algebraically closed.
While the projective dimension of the bimodule $A$ is known to be equal to the global dimension of the algebra for algebraically closed fields, it seems that the projective dimension of $D(A)$ as a bimodule is not known and there is no homological description for $\pd_{A^e}(D(A))$ or equivalently $\id_{A^e}(A)$ in the literature. This motivates us to pose the following question:
\begin{question}
Let $A$ be a finite dimensional algebra. Is there a nice homological description of the projective dimension of $D(A)$ as a bimodule, or equivalently of the injective dimension of $A$ as a bimodule?
\end{question}
We give a quick example that shows that the projective dimension of $D(A)$ can be equal to the global dimension, but this is not true in general.
\begin{example}
The Nakayama algebra with Kupisch series [2,1] has global dimension 1 and $D(A)$ has projective dimension 1 as a bimodule. For the Nakayama algebra with Kupisch series [2,3] the bimodule $D(A)$ has projective dimension 4, while the algebra has global dimension 2.

\end{example}

The following conjecture is known as the Nakayama conjecture and is stated for the first time by Nakayama in 1958, see \cite{Nak}.
\begin{conjecture} \label{Nakayama conjecture}
A finite dimensional algebra $A$ is selfinjective if and only if $A$ has infinite dominant dimension.
\end{conjecture} 

Our main result gives a new viewpoint on the Nakayama conjecture using the $\mho$-quiver of Ringel and Zhang, that we state as a corollary.

\begin{corollary}
The following are equivalent for a finite dimensional algebra $A$.
\begin{enumerate}
\item $A$ has infinite dominant dimension.
\item $A$ is $\infty$-torsionfree.
\item $A$ is the end of an infinite path in the $\mho$-quiver of $A^e$.

\end{enumerate}

\end{corollary} 
\begin{proof}
The equivalence of (1) an (2) follows directly by our main result \ref{mainresult}, while the equivalence of (2) and (3) is a consequence of \ref{ringelzhangpropo} (4).

\end{proof}

\section{The Gorenstein bimodule conjecture}

We first recall some definitions from Gorenstein homological algebra.
Recall that a finite dimensional algebra $A$ is called \emph{Gorenstein} in case the left and right injective dimension of the regular $A$-module coincide and are finite.
An $A$-module $M$ is called \emph{Gorenstein projective} in case $\Ext_A^i(M,A)=0=\Ext_A^i(\Tr(M),A)=0$ for all $i \geq 1$.
Following \cite{RZ} an algebra $A$ is called \emph{left weakly Gorenstein} in case every module $M$ with $\Ext_A^i(M,A)=0$ for $i>0$ is Gorenstein projective.
We give the following new Gorenstein homological conjecture, called the Gorenstein bimodule conjecture, that we will motivate afterwards by relating it to the Nakayama and Tachikawa conjectures:
\begin{conjecture} \label{gorbimodconjecture}
A finite dimensional algebra $A$ is selfinjective if and only if $A$ as a bimodule is Gorenstein projective.

\end{conjecture}
The conjecture is known to be true in case $A$ is Gorenstein by results of Shen in \cite{S} and more generally when $A$ is left weakly Gorenstein by results in \cite{Mar}.

The next conjecture is called the first Tachikawa conjecture and was first stated in the book \cite{Ta}. 
\begin{conjecture} \label{Tachikawa conjecture}
A finite dimensional algebra $A$ is selfinjective if and only if $\Ext_A^i(D(A),A)=0$ for all $i \geq 1$.
\end{conjecture}

It is known that the truth of the Nakayama conjecture for all algebras $A$ would imply the Tachikawa conjecture, but it is not known whether the truth of the Nakayama conjecture for a fixed algebra implies also that the first Tachikawa conjecture is true for this fixed algebra.
Recall that the \emph{finitistic dimension} of a finite dimensional algebra $A$ is defined as the supremum of all projective dimensions of modules having finite projective dimension. The finitistic dimension conjecture states that all finite dimensional algebras have finite finitistic dimension and it is known that the finitistic dimension conjecture implies the Nakayama conjecture for a given algebra $A$, see for example \cite{Yam} for more on those conjectures and their relation to other homological conjectures.

As an application of our main result of this article we can relate the Gorenstein bimodule conjecture to the Nakayama and first Tachikawa conjectures:
\begin{theorem} \label{gorbimodtheorem}
Let $A$ be a finite dimensional algebra.
Then the following are equivalent for $A$:
\begin{enumerate}
\item The Gorenstein bimodule conjecture holds for $A$.
\item The Nakayama conjecture or the first Tachikawa conjecture holds for $A$.

\end{enumerate}

\end{theorem}
\begin{proof}
This follows directly from our main result \ref{mainresult} by noting that $A$ as a bimodule is $\infty$-torsionfree if and only if $A$ has infinite dominant dimension and $\Ext_{A^e}^i(A,A^e) \cong \Ext_A^i(D(A),A)$ for all $i >0$ by lemma 2.1. of \cite{Mar}.

\end{proof}

\begin{corollary}
Let $A$ be an algebra of finite finitistic dimension, then $A$ satisfies the Gorenstein bimodule conjecture.

\end{corollary}
\begin{proof} 
Since having finite finitistic dimension implies that $A$ satisfies the Nakayama conjecture, 
by \ref{gorbimodtheorem} $A$ also satisfies the Gorenstein bimodule conjecture. 
\end{proof}
In particular, the Gorenstein bimodule conjecture is true for Gorenstein, monomial or local algebras since all such algebras have finite finitistic dimension.

For gendo-symmetric algebras the Gorenstein bimodule conjecture is even equivalent to the Nakayama conjecture, we state this is a corollary in the following form:

\begin{corollary}
Let $A$ be a gendo-symmetric algebra.
Then the following are equivalent:
\begin{enumerate}
\item $A$ is Gorenstein projective as an $A$-bimodule.
\item $A$ has infinite dominant dimension.

\end{enumerate}

\end{corollary}
\begin{proof}
This follows directly from the fact that for gendo-symmetric algebras, the dominant dimension of $A$ is given by $\domdim(A) = \inf \{ i \geq 1 | \Ext_A^i(D(A),A) \neq 0 \}+1$, see \ref{Fankoetheorem} (2) (and that $\Ext_{A^e}^i(A,A^e) \cong \Ext_A^i(D(A),A)$ for all $i >0$), and by our main result \ref{mainresult} that the dominant dimension is infinite if and only if $A$ as a bimodule is $\infty$-torsionfree.

\end{proof}

\section*{Acknowledgments} 
Rene Marczinzik is funded by the DFG with the project number 428999796.

\end{document}